\documentclass[11pt,twoside,reqno,letter]{amsart}
\usepackage[margin=1in]{geometry}
\usepackage{amsmath,graphicx,varioref,amscd,amssymb,color,bm,amsthm,epsfig,color,enumerate,fancybox,bbm,esint,upref,xcolor,latexsym,mathtools,breqn,appendix,setspace,physics,tikz,pdfsync,scalerel}
\usepackage[pdftex, colorlinks=true, linktocpage=true]{hyperref}

\makeatletter
\def\@tocline#1#2#3#4#5#6#7{\relax
	\ifnum #1>\c@tocdepth 
	\else
	\par \addpenalty\@secpenalty\addvspace{#2}%
	\begingroup \hyphenpenalty\@M
	\@ifempty{#4}{%
		\@tempdima\csname r@tocindent\number#1\endcsname\relax
	}{%
		\@tempdima#4\relax
	}%
	\parindent\z@ \leftskip#3\relax \advance\leftskip\@tempdima\relax
	\rightskip\@pnumwidth plus4em \parfillskip-\@pnumwidth
	#5\leavevmode\hskip-\@tempdima
	\ifcase #1
	\or\or \hskip 1em \or \hskip 2em \else \hskip 3em \fi%
	#6\nobreak\relax
	\dotfill\hbox to\@pnumwidth{\@tocpagenum{#7}}\par
	\nobreak
	\endgroup
	\fi}

\def\l@subsection{\@tocline{2}{0pt}{2pc}{5pc}{}}
\makeatother

\numberwithin{equation}{section}

\theoremstyle{plain}
\newtheorem{thm}{Theorem}[section]
\newtheorem*{thm*}{Theorem}
\newtheorem{mydef}[thm]{Definition}
\newtheorem{lem}[thm]{Lemma}
\newtheorem*{lem*}{Lemma}

\newtheorem*{prop*}{Proposition}

\theoremstyle{remark}
\newtheorem{rem}[thm]{Remark}

\newcommand{\R}{\mathbb{R}}

\newcommand{\N}{\mathbb{N}}
\newcommand{\Z}{\mathbb{Z}}
\newcommand{\T}{\mathbb{T}}
\newcommand{\Leray}{\mathcal{P}}
\newcommand{\wh}{\widehat}

\newcommand{\mi}{m_i}

\newcommand{\mf}{m_f}

\newcommand{\us}{u_{\text{\scriptsize s}}}
\newcommand{\un}{u_{\text{\scriptsize n}}}
\newcommand{\ws}{\omega_{\text{\scriptsize s}}}
\newcommand{\wn}{\omega_{\text{\scriptsize n}}}
\newcommand{\Ws}{\Omega_{\text{\scriptsize s}}}
\newcommand{\Wn}{\Omega_{\text{\scriptsize n}}}

\DeclarePairedDelimiter{\ceil}{\lceil}{\rceil}

\begin{document}
	
	
	\title[Local-in-time analytic solutions for an inviscid model of superfluidity in 3D]{Local-in-time analytic solutions for an inviscid model of superfluidity in 3D}
	
	\author[Jayanti]{Pranava Chaitanya Jayanti}
	\address[Jayanti]{\newline
		Department of Mathematics \\ University of Southern California \\ Los Angeles, CA 90089, USA}
	\email[]{\href{pjayanti@usc.edu}{pjayanti@usc.edu}}	
	
	\date{\today}
	
	
	

    \begin{abstract}
		We address the existence of solutions for the inviscid version of the Hall-Vinen-Bekharevich-Khalatnikov equations in 3D, a macro-scale model of superfluidity. This system couples the incompressible Euler equations for the normal fluid and superfluid using a nonlinear mutual friction term that acts only at points of non-zero superfluid vorticity. In the first rigorous study of the inviscid HVBK system, we construct a unique local-in-time solution that is analytic in time and space.
	\end{abstract}
 
	\maketitle

	\setcounter{tocdepth}{2} 
	
	
	\section{Introduction} \label{sec:intro}

    Isobaric cooling (at pressures below $25$ bar) of helium-4 results in a second-order phase transition that can be explained only using quantum mechanics (Bose-Einstein condensation). Below $2.17$K, there is a mixture of two different phases of helium: a run-of-the-mill normal fluid and a strange superfluid phase. The number of He-4 atoms in the latter increases as the temperature is lowered, until it is entirely superfluid at absolute zero. Experimental evidence suggests that the two phases are unlike standard multiphase flows in that they do not possess any distinct phase boundaries. Instead, both fluids occupy the entire domain of the flow, interacting with each other only at points where the superfluid vorticity is non-zero~\cite{Vinen:808382,Vinen2006AnTurbulence,Landau1941TheoryII}. This has prompted descriptions of superfluidity at various length scales: micro-scale quantum mechanical models~\cite{Pitaevskii1959PhenomenologicalPoint,Khalatnikov1969AbsorptionPoint,Carlson1996AVortices}, meso-scale vortex-based models~\cite{Schwarz1978TurbulenceCounterflow,Schwarz1985Three-dimensionalInteractions,Schwarz1988Three-dimensionalTurbulence}, and macro-scale classical fluid dynamical models~\cite{Holm2001,Roche2009QuantumCascade,Tchoufag2010EddyII,Salort2011MesoscaleTurbulence,Verma2019TheModel}. It is to be remarked that a unified description of superfluidity is still wanting; see ~\cite{Barenghi2001QuantizedTurbulence,Paoletti2011QuantumTurbulence,Barenghi2014ExperimentalFluid,Berloff2014ModelingTemperatures,Jayanti2022AnalysisSuperfluidity} and references therein for more details on these various models. 
    
    One of the non-classical effects in superfluids is the existence of quantized vortex filaments, which is the basis for the meso-scale models. However, at length scales much larger than the inter-vortex spacing, the discrete nature of the superfluid vorticity can be ignored, and a classical picture emerges. Known as the Hall-Vinen-Bekharevich-Khalatnikov (HVBK) equations, this macro-scale model describes the superfluid using the Euler equations and the normal fluid using the Navier-Stokes equations, and couples them with a nonlinear mutual friction term that vanishes at points of zero superfluid vorticity. In this paper, we present the space-time analyticity of the \textit{inviscid} incompressible HVBK equations in 3D, i.e., both fluids are governed by the incompressible Euler equations. The study of the equations of fluid dynamics in analytic classes (in the space variable) was pioneered by Foias and Temam~\cite{Foias1989GevreyEquations}. A straightforward consequence of such high regularity in the spatial variable is the time analyticity of the solutions~\cite{Constantin1988Navier-StokesEquations,Promislow1991TimeEquations}. The analyticity of $L^p$ solutions to the incompressible Navier-Stokes equations has also been investigated in various works~\cite{Grujic1998SpaceLp,Lemarie-Rieusset2000UneR3,Lemarie-Rieusset2004NouvellesR3}.  Analyticity for the case of inviscid equations was first shown for the ``great lake'' equations~\cite{Levermore1997AnalyticityEquation}, of which the incompressible Euler equations are a special case. This was achieved by trading the regularity of the solution for a dissipative term in the analytic norm energy balance. The method of Foias and Temam was extended to general dissipative parabolic PDEs with analytic nonlinearities~\cite{Ferrari1998GevreyEquations,Bae2014GevreyNonlinearity}, while the study of inviscid equations was further developed in~\cite{Kukavica2009OnEquations} and~\cite{Biswas2017SpaceDynamics}. The former provided improved estimates of the analyticity radius for the 3D Euler equations, whereas the latter extended the results of Levermore and Oliver to encompass a wide range of incompressible, inviscid equations. In~\cite{Constantin2016ContrastEquations,Camliyurt2018OnEquations}, the authors demonstrated the persistence of the analyticity radius of the Euler equations in Lagrangian coordinates, a property that was also shown not to hold in the Eulerian formulation. The machinery of analyticity was also recently adopted to the setting of compressible flows~\cite{Bae2020AnalyticityEquations,Charve2021GevreyCapillarity,Jang2021MachSpaces,Deng2022RadiusSystem,Li2023TheSpaces}. Our work in this paper is inspired primarily by~\cite{Biswas2017SpaceDynamics}, and the main challenge involves the mutual friction term between the superfluid and normal fluid, which is singular at points of vanishing superfluid vorticity. Thus, the central question is: Given a non-vanishing initial superfluid vorticity (in the analytic class), does it remain non-vanishing? We find the answer in the affirmative, at least locally in time. To the best of our knowledge, this result is the first of its kind.

    We now briefly outline the notation used in this article. Following this, in Section~\ref{sec:mathematical model and main results}, we present the HVBK equations in the form pertaining to our work, before defining analytic classes, and stating the main result. All of the required a priori estimates are derived in Section ~\ref{sec:a priori estimates}. Subsequently, the construction of the solutions is briefly discussed in Section~\ref{sec:existence of a unique, time-analytic solution}, along with the proof of uniqueness.	
	
    \subsection{Notation} \label{sec:notation}
	We denote by $H^s(\T^3)$ the completion of $C^{\infty}(\T^3)$ under the Sobolev norm $H^s$. 
    Consider a 3D vector-valued function $u\equiv (u_1,u_2,u_3)\in C^{\infty}(\T^3)$. The set of all divergence-free, smooth 3D functions $u$ defines $C^{\infty}_{\text{d}}(\T^3)$. Then, $H^s_{\text{d}}(\T^3)$ is the completion of $C^{\infty}_{\text{d}}(\T^3)$ under the $H^s$ norm. In particular, $H^0_{\text{d}}$ is denoted by $H$. The Leray projector, $\Leray\colon L^2\mapsto H$, projects the space of square-integrable vector-valued functions into its divergence-free subspace. The $L^2$ scalar product between H\"older conjugate vector-valued functions $v$ and $w$ is denoted by $\langle v,w \rangle = \int_{\T^3} v\cdot w \ dx$.
	
	
	We also use the notation $X\lesssim Y$ and $X\gtrsim Y$ to imply that there exists a positive constant $C$ such that $X\le CY$ and $CX\ge Y$, respectively. Throughout the article, $C$ is used to denote a possibly large constant that depends on the system parameters, and whose value can vary across the different steps of calculations.

    \section{Mathematical model and main result} \label{sec:mathematical model and main results}
    \subsection{The governing equations} \label{sec:the governing equations}
    We now present the form of the HVBK equations that is used here, namely
    \begin{equation} \label{eq:HVBK}
        \begin{aligned}
    		\partial_t \us + \us\cdot\nabla \us + \nabla p_{\text{\scriptsize s}}
            &= 
            -\rho_{\text{n}} \frac{\ws}{\abs{\ws}}\times (\ws\times v) \\
            \partial_t \un + \un\cdot\nabla \un + \nabla p_{\text{\scriptsize n}}
            &= 
            \rho_{\text{s}} \frac{\ws}{\abs{\ws}}\times (\ws\times v) \\
            v 
            &=
            \un - \us \\
            \nabla\cdot \us 
            &= 
            0 \\
            \nabla\cdot \un 
            &= 
            0 .
    	\end{aligned}
    \end{equation}
	Here, the subscripts $\text{s}$ and $\text{n}$ denote the superfluid and normal fluid, respectively. The velocity, pressure, and density fraction are labeled by $u,p$, and $\rho$ (respectively), and the vorticity is $\omega = \nabla\times u$. Since both fluids are incompressible, it is clear that the density fractions $\rho_{\text{\scriptsize s}}$ and $\rho_{\text{\scriptsize n}}$ are constant (with $\rho_{\text{\scriptsize s}}+\rho_{\text{\scriptsize n}} = 1$), and the pressures $p_{\text{\scriptsize s}}$ and $p_{\text{\scriptsize s}}$ are simply Lagrange multipliers that vanish under the action of the Leray projector $\Leray$, resulting in
    \begin{equation} \label{eq:HVBK Leray projected}
        \begin{aligned}
    		\partial_t \us + \Leray(\us\cdot\nabla \us)
            &= 
            -\rho_{\text{\scriptsize n}} \Leray\left(\frac{\ws}{\abs{\ws}}\times (\ws\times v)\right) \\
            \partial_t \un + \Leray(\un\cdot\nabla \un)
            &= 
            \rho_{\text{\scriptsize s}} \Leray\left(\frac{\ws}{\abs{\ws}}\times (\ws\times v)\right) .
    	\end{aligned}
    \end{equation}
    These equations are supplemented with the initial conditions
	\begin{equation} \label{eq:initial conditions}
		\us(0,x) 
        = 
        \us^0(x), 
        \quad 
        \un(0,x) 
        = 
        \un^0(x) 
        \quad 
        \text{a.e.} \ x\in\T^3 .
	\end{equation}
	We use periodic boundary conditions, i.e., we are working on a 3-dimensional torus $[0,2\pi]^3$. As a result, it is necessary that the average vorticity vanishes, i.e.,  $\int_{\T^3} \omega(t,x) = 0$, for both fluids. 
 
    Compared to the original model~\cite{Holm2001}, the system in~\eqref{eq:HVBK} has excluded compressibility and thermal effects, as well as a higher-order ``vortex tension'' term that is expected to lead to ill-posedness. An additional term in the mutual friction, which makes no qualitative difference to the results here, has been disregarded for simplicity. Finally, we treat the normal fluid as inviscid, in contrast to earlier mathematical work by the author~\cite{Jayanti2021GlobalEquations}, where global well-posedness of the 2D HVBK system was established if both fluids are considered viscous. In the physics community, the HVBK equations are often used to study superfluid Couette flow~\cite{Barenghi1988TheII,Henderson2004SuperfluidAnnulus} and superfluid turbulence. Particularly in the context of turbulence, different schemes have been employed to handle the generation of small scales in simulations of the inviscid superfluid: sub-grid scale modeling~\cite{Tchoufag2010EddyII}, spectral suppression of small scales~\cite{Salort2011MesoscaleTurbulence}, and adding an artificial viscosity~\cite{Roche2009QuantumCascade,Verma2019TheModel}. In reality, there is another mechanism for energy loss through nonlinear interactions of Kelvin waves along the vortex filaments, which lead to an energy cascade and eventually phononic dissipation, but that is not described by this macro-scale model.

    The system in~\eqref{eq:HVBK} (and equivalently~\eqref{eq:HVBK Leray projected}) is constructed so as to conserve mass and momentum, and also obey an energy balance equation. Momentum conservation can be seen by multiplying the superfluid equation by $\rho_{\text{\scriptsize s}}$, the normal fluid equation by $\rho_{\text{\scriptsize n}}$, adding them, and integrating over space. Using the periodic boundary conditions, we observe that the total momentum of the two fluids, i.e., $\int_{\T^3} (\rho_{\text{\scriptsize s}} \us(t) + \rho_{\text{\scriptsize n}} \un(t))$, is constant. Energy conservation also follows similarly: Test the two equations by (the corresponding) $\rho u$ and add to (formally) obtain
    \begin{equation} \label{eq:energy estimate}
        \frac{1}{2}\frac{d}{dt}\left(\rho_{\text{\scriptsize s}}\norm{\us}_{L^2}^2 + \rho_{\text{\scriptsize n}}\norm{\un}_{L^2}^2\right) + \rho_{\text{\scriptsize s}} \rho_{\text{\scriptsize n}}\int_{\T^3}\left(\lvert\ws\rvert\lvert v\rvert^2 -\frac{\left(\ws\cdot v\right)^2}{\abs{\ws}} \right) 
        = 
        0 .
    \end{equation}
    Since the integrand in the second term is always non-negative, we see that the total energy of the system is bounded. The mutual friction acts as a retarding force, working to equilibrate the two phases. Its structure is motivated by the experimental observation that the fluids only interact at points of non-zero superfluid vorticity. The singularity that stems from the magnitude of the vorticity vector should be acknowledged, as it constitutes the main challenge in the study of this system. We circumvent this hurdle by ensuring that no regions of vanishing superfluid vorticity occur. Since we seek smooth (analytic) solutions, we begin from initial data that satisfies $\inf_{x\in\T^3} \abs{\ws^0}=\mi>0$, and limit the evolution of the system to a time when $\inf_{x\in\T^3} \abs{\ws} \ge \mf$, where $\mf<\mi$. Therefore, the solution we obtain is local-in-time. A similar strategy was employed in the analysis of a micro-scale model of superfluidity in~\cite{Jayanti2022LocalSuperfluidity,Jang2024Small-dataSuperfluidity,Jang2024OnModel}, where it was necessary to prevent the formation of vacuum states.
    
    \subsection{Analytic classes} \label{sec:Gevrey classes}
    In the case of the incompressible Navier-Stokes equations (NSE), the average velocity is a conserved quantity so that we may set it to be zero (using a Galilean transformation). This, combined with the Poincar\'e inequality, allows one to work with the prevailing dissipative Stokes operator. Absent such a luxury in our system, we resort to the method employed in~\cite{Ferrari1998GevreyEquations} and consider the operator $A:= I-\Delta$, which accounts for the inhomogeneous norms that are needed. Before defining the analytic classes, we remark that the Fourier decomposition of any vector-valued function $f\in H^s(\T^3)$ is denoted by $(\wh{f}(k))_{k\in\Z^3}$.

    \begin{mydef}[Analytic class]
        For any $p\ge 0$ and $\sigma\ge 0$, the corresponding analytic class is defined~as 
        \begin{equation} \label{eq:Gevrey class definition}
            G^{\frac{p}{2}}_{\sigma}(\T^3):= D(A^{\frac{p}{2}}e^{\sigma A^{\frac{1}{2}}}) ,
        \end{equation}
        which is a Hilbert space with the inner product
        \begin{equation} \label{eq:Gevrey class inner product}
            \langle v,w \rangle_{G^{\frac{p}{2}}_{\sigma}} := \sum_{j\in\Z^3} \wh{v}(j)\cdot \overline{\wh{w}(j)} (1+\abs{j}^2)^p e^{2\sigma (1+\abs{j}^2)^{\frac{1}{2}}} ,
        \end{equation}
        (where $\overline{\wh{w}(j)}$ denotes the complex conjugate of the Fourier coefficients) and the norm
        \begin{equation} \label{eq:Gevrey class norm}
            \lVert v \rVert_{G^{\frac{p}{2}}_{\sigma}}^2 :=  \sum_{k\in\Z^3} \lvert\wh{v}(k)\rvert^2 (1+\abs{k}^2)^p e^{2\sigma (1+\abs{k}^2)^{\frac{1}{2}}} .
        \end{equation}
    \end{mydef}
    \begin{rem} \label{rem:properties of Gevrey norm}
        The inclusion $G^{\frac{p}{2}}_{\sigma}\subset H^{p}$ holds for all $p\ge 0$ and $\sigma> 0$. We also note that $G^{\frac{p}{2}}_{\sigma}$ forms an algebra for $p>\frac{3}{2}$, and the proof can be found in~\cite[Lemma 1]{Ferrari1998GevreyEquations}.
    \end{rem}
 
	\subsection{The main result} \label{sec:the main result}
    We are now ready to state the main theorem. Throughout the article, we assume that $\rho_{\text{\scriptsize s}}$ and $\rho_{\text{\scriptsize n}}$ are fixed constants, and as a consequence they do not explicitly feature in the estimates or results.
    \begin{thm} [Local-in-time existence and uniqueness] \label{thm:Gevrey regularity}
        For $p>\frac{5}{2}$, $\mi>0$, and $0< \sigma_0 < \frac{\mi}{2C_0}$ (where $C_0$ is the Poincar\'e constant), suppose that the initial velocities $u_{\text{\normalfont \scriptsize s}}^0,u_{\text{\normalfont \scriptsize n}}^0$ belong to $(G^{\frac{p+1}{2}}_{\sigma_0}\cap H)(\T^3)$, with the initial superfluid vorticity $\omega_{\text{\normalfont \scriptsize s}}^0 = \nabla\times u_{\text{\normalfont \scriptsize s}}^0$ satisfying
        \begin{equation} \label{eq:initial data conditions}
            \inf_{x\in\T^3} \abs{\omega_{\text{\normalfont \scriptsize s}}^0} = \mi .
        \end{equation}
        Then, for any $2C_0 \sigma_0<\mf<\mi$, there exists a $\delta > 0$ and a time $T^*$ (both depending only on the size of the initial data, $\mi$, $\mf$, $\sigma_0$, and $p$) such that
        \begin{itemize}
            \item the system~\eqref{eq:HVBK Leray projected} admits a unique solution $u_{\text{\normalfont \scriptsize s}}(t),u_{\text{\normalfont \scriptsize n}}(t)\in (G^{\frac{p+1}{2}}_{\sigma(t)}\cap H)(\T^3)$, with $\sigma(t) = \sigma_0 - \delta t$, for $t\in[0,T^*)$, 
            \item the solution satisfies $\inf_{t,x\in (0,T^*)\times\T^3} \abs{\omega_{\text{\normalfont \scriptsize s}}(t,x)} \ge \mf$, and
            \item the map $t\mapsto u(t)$ is analytic on $(0,T^*)$.
        \end{itemize}
	\end{thm}
    
    To prove Theorem ~\ref{thm:Gevrey regularity}, we work with a Galerkin-truncated approximation of~\eqref{eq:HVBK}. The strategy involves two distinct parts: deriving the required a priori estimates in the analytic norm, and ensuring that $\abs{\ws}$ is bounded below by $\mf$. The existence of a solution to the approximate system follows from the ODE theory, so we only present a sketch. It is also argued that a standard energy-based approach to uniqueness is not sufficient, and the estimates are instead carried out in the analytic norms. Finally, the time analyticity of the solution is a consequence of the analytic regularity in space, as shown in~\cite[Theorem 12.1]{Constantin1988Navier-StokesEquations} or~\cite[Theorem 1.1]{Promislow1991TimeEquations}. 

    \begin{rem} \label{rem:effect of changing mi and mf}
        We note that the regularity of the initial data increases with $\sigma_0$. From Theorem~\ref{thm:Gevrey regularity}, we see that if the initial lower bound to the vorticity magnitude $(\mi)$ is higher, then we can also start from initial data that is more regular. Similarly, if we wish to extend the evolution time by decreasing $\mf$, then it comes at the expense of a reduced regularity of the initial data. 
    \end{rem}
    
    In our system, we have several nonlinear terms which are treated using the following lemma that generalizes~\cite[Lemma~8]{Levermore1997AnalyticityEquation} and~\cite[Proposition~7.2]{Biswas2017SpaceDynamics}.
    \begin{lem}[Nonlinear estimate] \label{lem:nonlinear estimate}
        Consider an operator $T$ whose representation in Fourier space is in the form of a multiplier, i.e., 
        \begin{equation} \label{eq:Fourier multiplier operator}
            \wh{T(f)}(k) := m(k)\wh{f}(k) , \qquad \lvert m(k) \rvert \lesssim \lvert k \rvert.
        \end{equation}
        Then, for $1\le i\le K\in\N$ and $f_i,g\in C^{\infty}_x$, we have
        \begin{equation}
            \lvert \langle \Pi_{i=1}^K f_i , Tg\rangle_{G^{\frac{p}{2}}_{\sigma}} \rvert 
            \lesssim 
            \sum_{j=1}^K \lVert A^{\frac{1}{4}} f_j \rVert_{G^{\frac{p}{2}}_{\sigma}} \left( \Pi_{\substack{i=1 \\ i\neq j}}^K \lVert f_i \rVert_{G^{\frac{p}{2}}_{\sigma}} \right) \lVert A^{\frac{1}{4}} g \rVert_{G^{\frac{p}{2}}_{\sigma}} ,
        \end{equation}
        for $p>2$.
    \end{lem}
    
    \begin{proof}
        We begin by noting two results. The first is the elementary inequality 
        \begin{equation} \label{eq:elementary inequality}
            \left(\sum_{i=1}^{K} x_i\right)^r \lesssim \sum_{i=1}^K x_i^r ,
        \end{equation}
        valid for any $x_i \ge 0$ and $r>0$. The second is an embedding result which states
        \begin{equation} \label{eq:embedding of Wiener norm}
            \Vert f\Vert_{L^{\infty}} \le \sum_{k}\vert \wh{f}(k)\vert \lesssim \Vert A^{\frac{s}{2}} f \Vert_{L^2} \sim \Vert f \Vert_{H^s} ,
        \end{equation}
        for any $s>\frac{3}{2}$. It follows from the Fourier series expansion, the Cauchy-Schwarz inequality and the sufficiently fast decay of $(1+\vert k\vert^2)^{-s}$. We also record that the space of functions $f$ satisfying $\sum_k \vert \wh{f}(k)\vert < \infty$ is called the Wiener algebra, denoted by $\mathcal{W}$. In what follows, we denote by $I$ the set of three-dimensional vectors satisfying $\sum_{i=1}^K h_i + k = 0$.

        We now proceed to prove the lemma, as
        \begin{equation*}
        \begin{aligned}
            &\lvert \langle \Pi_{i=1}^K f_i , Tg\rangle_{G^{\frac{p}{2}}_{\sigma}} \rvert \\
            &\lesssim \sum_I \left(\Pi_{i=1}^K \lvert \wh{f_i}(h_i)\rvert\right) (1+\lvert k\rvert^2)^p e^{2\sigma (1+\lvert k\rvert^2)^{\frac{1}{2}}} \lvert k \rvert \lvert \wh{g}(k) \rvert \\
            &\lesssim \sum_I \sum_{j=1}^K (1+\lvert h_j\rvert^2)^{\frac{p}{2}} \left(\Pi_{i=1}^K e^{\sigma (1+\lvert h_i\rvert^2)^{\frac{1}{2}}} \lvert \wh{f_i}(h_i)\rvert\right) (1+\lvert k\rvert^2)^{\frac{p}{2}} e^{\sigma (1+\lvert k\rvert^2)^{\frac{1}{2}}} \lvert k \rvert \lvert \wh{g}(k) \rvert \\
            &\lesssim \sum_I \sum_{j=1}^K (1+\lvert h_j\rvert^2)^{\frac{p}{2}} \left(\Pi_{i=1}^K (1+\lvert h_i\rvert^2)^{\frac{1}{4}} e^{\sigma (1+\lvert h_i\rvert^2)^{\frac{1}{2}}} \lvert \wh{f_i}(h_i)\rvert\right) (1+\lvert k\rvert^2)^{\frac{p}{2}+{\frac{1}{4}}} e^{\sigma (1+\lvert k\rvert^2)^{\frac{1}{2}}} \lvert \wh{g}(k) \rvert \\
            &\lesssim \sum_{j=1}^K \lVert A^{\frac{1}{4}} f_j \rVert_{G^{\frac{p}{2}}_{\sigma}} \left(\Pi_{\substack{i=1 \\ i\neq j}}^K \lVert A^{\frac{1}{4}} e^{\sigma A^{\frac{1}{2}}} f_i \rVert_{\mathcal{W}} \right) \lVert A^{\frac{1}{4}} g \rVert_{G^{\frac{p}{2}}_{\sigma}} \\
            &\lesssim \sum_{j=1}^K \lVert A^{\frac{1}{4}} f_j \rVert_{G^{\frac{p}{2}}_{\sigma}} \left( \Pi_{\substack{i=1 \\ i\neq j}}^K \lVert f_i \rVert_{G^{\frac{p}{2}}_{\sigma}} \right) \lVert A^{\frac{1}{4}} g \rVert_{G^{\frac{p}{2}}_{\sigma}} .
        \end{aligned}
        \end{equation*}
        In the second inequality, we used the triangle inequality and~\eqref{eq:elementary inequality} to write
        \begin{equation*}
            (1+\lvert k\rvert^2)^r \lesssim \sum_{j=1}^K (1+\lvert h_j\rvert^2)^r .
        \end{equation*}
        The fourth line follows similarly, since
        \begin{equation*}
            \lvert k\rvert^{\frac{1}{2}} = (\lvert k\rvert^2)^{\frac{1}{4}} \lesssim \left(\sum_{i=1}^K \lvert h_i\rvert^2 \right)^{\frac{1}{4}} \lesssim \Pi_{i=1}^K (1+\lvert h_i\rvert^2)^{\frac{1}{4}} .
        \end{equation*}
        An application of H\"older's inequality leads to the fifth line, while the final inequality is obtained by choosing $p>2$ and using the embedding in~\eqref{eq:embedding of Wiener norm}.
    \end{proof}

    \section{A priori estimates} \label{sec:a priori estimates}
    In this section, we derive formal estimates for the system in~\eqref{eq:HVBK Leray projected}, assuming that $t\in [0,T^*]$, where $T^*$ is the local existence time. The time $T^* := \min\{T_1,T_2\}$, where $T_1$ is the existence time arising from the energy estimate (see~\eqref{eq:T1 definition}), and 
    \begin{equation} \label{eq:T2 definition}
        T_2 := \inf \{t \ \rvert \inf_{x\in\T^3} \abs{\ws(t)}=\mf \} .
    \end{equation}

    \subsection{Approximate system} \label{sec:approximate system}
    We begin by considering a finite-dimensional approximation of~\eqref{eq:HVBK Leray projected}, via a Galerkin truncation. Indeed, we have
    \begin{equation} \label{eq:HVBK approximation}
        \begin{aligned}
    		\partial_t \us^N + P^N(\us^N\cdot\nabla \us^N)
            &= 
            -\rho_{\text{\scriptsize n}} P^N\left(\frac{\ws^N}{\abs{\ws^N}}\times (\ws^N\times v^N)\right) \\
            \partial_t \un^N + P^N(\un^N\cdot\nabla \un^N)
            &= 
            \rho_{\text{\scriptsize s}} P^N\left(\frac{\ws^N}{\abs{\ws^N}}\times (\ws^N\times v^N)\right) \\
            v^N 
            &=
            \un^N - \us^N .
    	\end{aligned}
    \end{equation}
    Here, $\us^N(t,x) = \sum_{\abs{k}\le N} \wh{\us}(t,k) e^{i k\cdot x}$ and $\ws^N(t,x) = \sum_{\abs{k}\le N} \wh{\ws}(t,k) e^{i k\cdot x}$ with $\wh{\ws}(t,k) = ik\times\wh{\us}(t,k)$, i.e., only the first $N$ Fourier modes of the velocity and vorticity are retained. The projector $P^N$ (the finite-dimensional restriction of $\Leray$) achieves this, while also ensuring that the $\wh{\us}(t,k)$ are three-dimensional vector-valued functions which satisfy the (spectral) divergence-free constraint: $k\cdot \wh{\us}(t,k) = 0$.

    Note that the nonlinear estimate derived in Lemma~\ref{lem:nonlinear estimate} is valid only if there is at most one derivative to be shared by the entire inner product. While the advection terms in~\eqref{eq:HVBK approximation} satisfy this requirement, this is not true of the mutual friction, which features multiple appearances from the superfluid vorticity (each of which represents a derivative loss compared to the velocity). Guided by this, we now reformulate~\eqref{eq:HVBK} so that it is more conducive to the methods used here. The identity $u = (-\Delta)^{-1}(\nabla\times\omega)$, translated into spectral form, reads $\wh{u}(k) = \frac{ik}{\abs{k}^2}\times\wh{\omega}(k)$, valid for any $k\neq 0$. As a consequence of the vorticity having zero mean value ($\wh{\omega}(0) = 0$) but not the velocity $(\wh{u}(0) \neq 0)$, it is not possible to entirely describe the system through just the vortex dynamics of the two fluids. To be more precise, we need an equation for the dynamics of the zeroth Fourier modes of the velocities, i.e., the (spatial) averages of the velocities. With this in mind, we arrive at an equivalent formulation of~\eqref{eq:HVBK approximation}, given by 
    \begin{equation} \label{eq:HVBK reformulated}
        \begin{aligned}
    		\partial_t \overline{\us^N}
            &= 
            -\rho_{\text{\scriptsize n}} \int_{\T^3}P^N\left(\frac{\ws^N}{\abs{\ws^N}}\times (\ws^N\times v^N)\right) \\
            \partial_t \overline{\un^N}
            &= 
            \rho_{\text{\scriptsize s}} \int_{\T^3}P^N\left(\frac{\ws^N}{\abs{\ws^N}}\times (\ws^N\times v^N)\right) \\
            \partial_t \ws^N
            &= 
            - P^N(\us^N\cdot\nabla \ws^N) + P^N(\ws^N\cdot\nabla \us^N) -\rho_{\text{\scriptsize n}} \nabla\times P^N \left(\frac{\ws^N}{\abs{\ws^N}}\times (\ws^N\times v^N)\right) \\
            \partial_t \wn^N 
            &= 
            - P^N(\un^N\cdot\nabla \wn^N) + P^N(\wn^N\cdot\nabla \un^N) + \rho_{\text{\scriptsize s}} \nabla\times P^N\left(\frac{\ws^N}{\abs{\ws^N}}\times (\ws^N\times v^N)\right) .
    	\end{aligned}
    \end{equation}
    The first and second equations are obtained by integrating~\eqref{eq:HVBK approximation} over $\T^3$, whereas the third and fourth equations are a result of applying the curl operator to ~\eqref{eq:HVBK approximation}. Henceforth, this is the system we analyze, supplemented with the initial conditions: $\overline{\us^N}(0,x) = P^N \overline{\us^0}(x)$ and $\ws^N(0,x) = P^N \ws^0(x)$, and likewise for the normal fluid as well. The calculations in the rest of this section are all valid in a local time interval $0\le t\le T^N$. In the course of deriving the a priori estimates, we establish that the local existence time can be chosen to be independent of $N$.

    \subsection{Average velocity estimate} \label{sec:average velocity estimate}
    Applying the Cauchy-Schwarz inequality to the RHS of~\eqref{eq:HVBK reformulated}$_1$,
    \begin{equation} \label{eq:Cauhcy-Schwarz on average superfluid equation}
        \partial_t \overline{\us^N} 
        \lesssim 
        \lVert\ws^N\rVert_{L^2} \lVert v^N\rVert_{L^2}
        \lesssim
        \lVert\ws^N\rVert_{L^2} ,
    \end{equation}
    where the second inequality follows from the energy estimate in~\eqref{eq:energy estimate}, since 
    \begin{equation*}
        \lVert v^N\rVert_{L^2} \le \lVert \us^N\rVert_{L^2} + \lVert \un^N\rVert_{L^2} \lesssim \lVert P^N \us^0\rVert_{L^2} + \lVert P^N \un^0\rVert_{L^2} \le \lVert \us^0\rVert_{L^2} + \lVert \un^0\rVert_{L^2} .
    \end{equation*}
    Integrating~\eqref{eq:Cauhcy-Schwarz on average superfluid equation} in time yields,
    \begin{equation} \label{eq:bound for average super velocity}
        \wh{\us}(t,0) \le \wh{\us^0}(0) + C\lVert\ws^N\rVert_{G^{\frac{p}{2}}_{\sigma}}T^N .
    \end{equation}
    Following a similar argument for~\eqref{eq:HVBK reformulated}$_2$, we arrive at
    \begin{equation} \label{eq:bound for average normal velocity}
        \wh{\un}(t,0) \le \wh{\un^0}(0) + C\lVert\ws^N\rVert_{G^{\frac{p}{2}}_{\sigma}}T^N .
    \end{equation}

    \begin{rem}
        As a consequence of~\eqref{eq:Gevrey class norm} and ~\eqref{eq:bound for average super velocity}, we also get the estimate (omitting the time variable for simplicity)
        \begin{equation} \label{eq:expressing Gevrey norm of u in terms of Gevrey norm of omega}
        \begin{aligned}
            \lVert A^{\frac{r}{2}} \us^N \rVert_{G^{\frac{p}{2}}_{\sigma}}^2 &= \sum_{\lvert k\rvert\le N} \abs{\wh{\us}(k)}^2 (1+\abs{k}^2)^{p+r} e^{2\sigma (1+\abs{k}^2)^{\frac{1}{2}}} \\
            &\le \abs{\wh{\us}(0)}^2 e^{2\sigma} + \sum_{1\le \lvert k\rvert\le N} \frac{1+\lvert k\rvert^2}{\abs{k}^2} \abs{\wh{\ws}(k)}^2 (1+\abs{k}^2)^{p-1+r} e^{2\sigma (1+\abs{k}^2)^{\frac{1}{2}}} \\
            &\lesssim \lvert\wh{\us^0}(0)\rvert^2 + \lVert A^{\frac{r}{2}} \ws^N \rVert_{G^{\frac{p-1}{2}}_{\sigma}}^2 (1+T^N)^2 ,
        \end{aligned}
        \end{equation}
        which will be used later.
    \end{rem}

    \subsection{Analytic norm energy estimate} \label{sec:Gevrey norm energy estimate}
    We now apply $A^{\frac{p}{2}}e^{\sigma(t)A^{\frac{1}{2}}}$ to~\eqref{eq:HVBK reformulated}$_1$, and then take the $L^2$ inner product with $A^{\frac{p}{2}}e^{\sigma(t)A^{\frac{1}{2}}}\ws^N$ to obtain
    \begin{equation*}
    \begin{aligned}
        \langle A^{\frac{p}{2}}e^{\sigma A^{\frac{1}{2}}}\ws^N, A^{\frac{p}{2}}e^{\sigma A^{\frac{1}{2}}}\partial_t \ws^N \rangle_{L^2_x} 
        &=
        -\langle A^{\frac{p}{2}}e^{\sigma A^{\frac{1}{2}}}\ws^N, A^{\frac{p}{2}}e^{\sigma A^{\frac{1}{2}}}P^N(\us^N\cdot\nabla \ws^N) \rangle_{L^2_x} \\
        &\quad +\langle A^{\frac{p}{2}}e^{\sigma A^{\frac{1}{2}}}\ws^N, A^{\frac{p}{2}}e^{\sigma A^{\frac{1}{2}}}P^N(\ws^N\cdot\nabla \us^N) \rangle_{L^2_x} \\
        &\quad - \rho_{\text{\scriptsize n}}\langle A^{\frac{p}{2}}e^{\sigma A^{\frac{1}{2}}}\ws^N, A^{\frac{p}{2}}e^{\sigma A^{\frac{1}{2}}}\nabla\times P^N\left(\frac{\ws^N}{\abs{\ws^N}}\times (\ws^N\times v^N)\right) \rangle_{L^2_x} ,
    \end{aligned}
    \end{equation*}
    where we have hidden the time dependence of $\sigma(t)$ for brevity. We integrate by parts (in time on the LHS and in space on the RHS), and use $\sigma(t) = \sigma_0 - \delta t$ for a $\delta>0$ that will be determined later, to arrive at
    \begin{equation} \label{eq:Gevrey energy estimate 1}
    \begin{aligned}
        \frac{1}{2}\frac{d}{dt}\lVert \ws^N \rVert_{G^{\frac{p}{2}}_{\sigma}}^2 + \delta\lVert A^{\frac{1}{4}} \ws^N \rVert_{G^{\frac{p}{2}}_{\sigma}}^2 
        &=
        \langle \nabla\ws^N, \us^N\otimes \ws^N \rangle_{G^{\frac{p}{2}}_{\sigma}} -\langle \nabla\ws^N, \ws^N\otimes \us^N \rangle_{G^{\frac{p}{2}}_{\sigma}} \\
        &\quad + \rho_{\text{\scriptsize n}}\langle \nabla\times\ws^N, \frac{\ws^N}{\abs{\ws^N}}\times (\ws^N\times v^N) \rangle_{G^{\frac{p}{2}}_{\sigma}} .
    \end{aligned}
    \end{equation}

    The first and second terms of the RHS can be treated using Lemma~\ref{lem:nonlinear estimate} with $K=2$. Indeed, the first term leads to
    \begin{equation} \label{eq:first term on RHS Gevrey energy estimate}
    \begin{aligned}
        \lvert\langle \nabla\ws^N, \us^N\otimes \ws^N \rangle_{G^{\frac{p}{2}}_{\sigma}}\rvert 
        &\lesssim 
        \lVert \us^N \rVert_{G^{\frac{p}{2}}_{\sigma}} \lVert A^{\frac{1}{4}} \ws^N \rVert_{G^{\frac{p}{2}}_{\sigma}}^2 + \lVert A^{\frac{1}{4}} \us^N \rVert_{G^{\frac{p}{2}}_{\sigma}} \lVert \ws^N \rVert_{G^{\frac{p}{2}}_{\sigma}} \lVert A^{\frac{1}{4}} \ws^N \rVert_{G^{\frac{p}{2}}_{\sigma}} \\
        &\lesssim \left(\lvert\wh{\us^0}(0)\rvert + \lVert \ws^N \rVert_{G^{\frac{p-1}{2}}_{\sigma}} (1+T^N)\right) \lVert A^{\frac{1}{4}} \ws^N \rVert_{G^{\frac{p}{2}}_{\sigma}}^2 \\
        &\quad + \left(\lvert\wh{\us^0}(0)\rvert + \lVert A^{\frac{1}{4}} \ws^N \rVert_{G^{\frac{p-1}{2}}_{\sigma}} (1+T^N)\right) \lVert \ws^N \rVert_{G^{\frac{p}{2}}_{\sigma}} \lVert A^{\frac{1}{4}} \ws^N \rVert_{G^{\frac{p}{2}}_{\sigma}} \\
        &\lesssim \left(\lvert\wh{\us^0}(0)\rvert + \lVert \ws^N \rVert_{G^{\frac{p}{2}}_{\sigma}} (1+T^N)\right) \lVert A^{\frac{1}{4}} \ws^N \rVert_{G^{\frac{p}{2}}_{\sigma}}^2 .
    \end{aligned}
    \end{equation}
    The second inequality results from~\eqref{eq:expressing Gevrey norm of u in terms of Gevrey norm of omega}, when $r=0$ and $r=\frac{1}{2}$.
    
    To deal with the third term on the RHS of~\eqref{eq:Gevrey energy estimate 1}, we apply Lemma~\ref{lem:nonlinear estimate} for $K=4$, which yields
    \begin{equation} \label{eq:last term on RHS Gevrey energy estimate}
        \begin{aligned}
            &\lvert \langle \nabla\times\ws^N, \frac{\ws^N}{\abs{\ws^N}}\times (\ws^N\times v^N) \rangle_{G^{\frac{p}{2}}_{\sigma}} \rvert \\
            &\lesssim \lVert A^{\frac{1}{4}} \ws^N \rVert_{G^{\frac{p}{2}}_{\sigma}} \left( \lVert A^{\frac{1}{4}} \ws^N \rVert_{G^{\frac{p}{2}}_{\sigma}} \lVert \ws^N \rVert_{G^{\frac{p}{2}}_{\sigma}} \lVert \frac{1}{\abs{\ws^N}} \rVert_{G^{\frac{p}{2}}_{\sigma}} \lVert v^N \rVert_{G^{\frac{p}{2}}_{\sigma}} \right. \\ 
            &\left. \qquad \qquad \qquad \qquad + \lVert \ws^N \rVert_{G^{\frac{p}{2}}_{\sigma}}^2 \lVert A^{\frac{1}{4}} \frac{1}{\abs{\ws^N}} \rVert_{G^{\frac{p}{2}}_{\sigma}} \lVert v^N \rVert_{G^{\frac{p}{2}}_{\sigma}} + \lVert \ws^N \rVert_{G^{\frac{p}{2}}_{\sigma}}^2 \lVert \frac{1}{\abs{\ws^N}} \rVert_{G^{\frac{p}{2}}_{\sigma}} \lVert A^{\frac{1}{4}} v^N \rVert_{G^{\frac{p}{2}}_{\sigma}} \right) \\
            &\lesssim \lVert A^{\frac{1}{4}} \ws^N \rVert_{G^{\frac{p}{2}}_{\sigma}}^2 \lVert \ws^N \rVert_{G^{\frac{p}{2}}_{\sigma}} \lVert A^{\frac{1}{4}} \frac{1}{\abs{\ws^N}} \rVert_{G^{\frac{p}{2}}_{\sigma}} \lVert A^{\frac{1}{4}} v^N \rVert_{G^{\frac{p}{2}}_{\sigma}} .
        \end{aligned}
    \end{equation}
    We now denote 
    \begin{equation} \label{defining X^N and Y^N}
    \begin{gathered}
        X^N := 1 + \lVert\ws^N \rVert_{G^{\frac{p}{2}}_{\sigma}}^2 + \lVert\wn^N \rVert_{G^{\frac{p}{2}}_{\sigma}}^2 , \quad Y^N := \lVert A^{\frac{1}{4}}\ws^N \rVert_{G^{\frac{p}{2}}_{\sigma}}^2 + \lVert A^{\frac{1}{4}}\wn^N \rVert_{G^{\frac{p}{2}}_{\sigma}}^2 , \\
        \overline{U} := \lvert\wh{\us^0}(0)\rvert + \lvert\wh{\un^0}(0)\rvert .
    \end{gathered}
    \end{equation} 
    We can use~\eqref{eq:expressing Gevrey norm of u in terms of Gevrey norm of omega}, as well as~\eqref{eq:1/omega analytic norm 3} for $r=\frac{1}{2}$, to rewrite~\eqref{eq:last term on RHS Gevrey energy estimate} as
    \begin{equation} \label{eq:last term on RHS Gevrey energy estimate 2}
    \begin{aligned}
        &\lvert \langle \nabla\times\ws^N, \frac{\ws^N}{\abs{\ws^N}}\times (\ws^N\times v^N) \rangle_{G^{\frac{p}{2}}_{\sigma}} \rvert \\
        &\lesssim \sqrt{X^N} \left( \overline{U} + \sqrt{X^N}(1+T^N) \right) \left( e^{\sigma_0} + (p+1)!\left(\frac{\mf}{2}-C_0\sigma_0\right)^{-p-2} \right) Y^N \\
        &\lesssim X^N (1+\overline{U}+T^N) Y^N .
    \end{aligned}
    \end{equation}
    Substituting~\eqref{eq:first term on RHS Gevrey energy estimate} and~\eqref{eq:last term on RHS Gevrey energy estimate 2} into~\eqref{eq:Gevrey energy estimate 1} therefore leads to
    \begin{equation} \label{eq:Gevrey energy estimate 2}
        \frac{1}{2}\frac{d}{dt}\lVert \ws^N \rVert_{G^{\frac{p}{2}}_{\sigma}}^2 + \delta\lVert A^{\frac{1}{4}}\ws^N \rVert_{G^{\frac{p}{2}}_{\sigma}}^2
        \lesssim X^N (1+\overline{U}+T^N) Y^N .
    \end{equation}
    
    Performing similar estimates on the normal fluid equation~\eqref{eq:HVBK approximation}$_2$, and adding the result with~\eqref{eq:Gevrey energy estimate 2} finally yields
    \begin{equation} \label{eq:Gevrey energy estimate 3}
        \frac{1}{2}\frac{d}{dt}X^N + \delta Y^N
        \le 
        CX^N (1+\overline{U}+T^N) Y^N .
    \end{equation}
    Then, for $X_0 := 1 + \lVert\ws^0 \rVert_{G^{\frac{p}{2}}_{\sigma_0}}^2 + \lVert\wn^0 \rVert_{G^{\frac{p}{2}}_{\sigma_0}}^2$, we have
    \begin{equation} \label{eq:delta definition}
    \begin{aligned}
        \delta := 2CX_0 (1+\overline{U}+T^N) \ge 2CX^N(0) (1+\overline{U}+T^N) ,
    \end{aligned}
    \end{equation}
    we have
    \begin{equation} \label{eq:Gevrey energy estimate result}
        X^N(t) + \delta\int_0^t Y^N(s) ds \le X_0 , \quad 0\le t < T^N .
    \end{equation}
    The existence time is defined by the relation
    \begin{equation} \label{eq:T1 definition}
        T^N := \frac{\sigma_0}{\delta} = \frac{\sigma_0}{2CX_0 (1+\overline{U}+T^N)} .
    \end{equation}
    The above equation admits a positive solution which depends only on $\mf$, $p$, $\sigma_0$, $\overline{U}$, $\lVert\ws^0 \rVert_{G^{\frac{p}{2}}_{\sigma_0}}$, and $\lVert\wn^0 \rVert_{G^{\frac{p}{2}}_{\sigma_0}}$. The existence time is thus independent of $N$, and so we label it $T_1$.
    
    In conclusion, from~\eqref{eq:expressing Gevrey norm of u in terms of Gevrey norm of omega} and~\eqref{eq:Gevrey energy estimate result}, we have $\us^N,\un^N \in L^{\infty}(0,T_1;G^{\frac{p+1}{2}}_{\sigma}\cap H) \cap L^2(0,T_1;G^{\frac{p+\frac{3}{2}}{2}}_{\sigma}\cap H)$, where the divergence-free property follows from the construction of the approximate velocity field. From here, we return to~\eqref{eq:HVBK approximation} to infer that $\partial_t \us^N,\partial_t \un^N \in L^{\infty}(0,T_1;G^{\frac{p}{2}}_{\sigma}\cap H) \cap L^2(0,T_1;G^{\frac{p+\frac{1}{2}}{2}}_{\sigma}\cap H)$. All of the regularity results mentioned here are uniform in $N$, which will allow us to extract a weakly convergent subsequence that retains the same properties.

    \subsection{Ensuring a positive lower bound to $\abs{\ws}$} \label{sec:Ensuring non-vanishing vorticity}
    We return to~\eqref{eq:HVBK reformulated}$_3$ and write its solution as
    \begin{equation} \label{eq:formal solution omega-s}
        \ws^N(t,x) = P^N \ws^0(x) + \int_0^t \mathcal{T}_{\text{\scriptsize s}}^N(\tau,x) d\tau ,
    \end{equation}
    where $\mathcal{T}_{\text{\scriptsize s}}^N$ is the RHS of~\eqref{eq:HVBK reformulated}$_3$. Observe that    \begin{equation} \label{eq:Bernstein inequality for complement of projection}
    \begin{aligned}
        \lVert (I-P^N)\ws^0 \rVert_{L^{\infty}} 
        &\lesssim N^{-s} \lVert \lvert\nabla\rvert^s (I-P^N)\ws^0 \rVert_{L^{\infty}} \\
        &\lesssim N^{-s} \lVert A^{\frac{s}{2}} (I-P^N)\ws^0 \rVert_{L^{\infty}} \lesssim N^{-s} \lVert \ws^0 \rVert_{G^{\frac{p}{2}}_{\sigma_0}} ,
    \end{aligned}
    \end{equation}
    for $s>0$ such that $\frac{3}{2}+s < p$. The first inequality can be proved similar to Bernstein's inequalities~\cite[Appendix A]{Tao2006NonlinearAnalysis}, while the last inequality is due to~\eqref{eq:embedding of Wiener norm}. Now, by the restriction on the initial vorticity, and from the triangle inequality, we know that
    \begin{equation} \label{eq:triangle inequality for vorticity}
        \mi \le \lvert \ws^0 \rvert \le \lvert P^N\ws^0 \rvert + \lvert (I-P^N)\ws^0 \rvert .
    \end{equation}
    Using~\eqref{eq:Bernstein inequality for complement of projection}, for a sufficiently large $N$, say $N>N^*$, we can ensure that the last term in~\eqref{eq:triangle inequality for vorticity} is less than, say $\frac{\mi-\mf}{2}$. This implies that, for $N>N^*$, we have $\lvert P^N\ws^0 \rvert \ge \frac{\mi+\mf}{2}$. Therefore, returning to~\eqref{eq:formal solution omega-s}, it is sufficient to demand that the time $T_2$ (defined in~\eqref{eq:T2 definition}) is such that
    \begin{equation} \label{eq:time integral of torque constraint}
        \left\lvert\int_0^{T_2} \mathcal{T}_{\text{\scriptsize s}}^N(\tau,x) d\tau \right\rvert \le \frac{\mi-\mf}{2} ,
    \end{equation}
    so that the triangle inequality yields the required condition $\lvert\ws^N(t,x)\rvert \ge \mf$. Using Sobolev embedding and~\eqref{eq:Gevrey class norm}, we can guarantee~\eqref{eq:time integral of torque constraint} with the sufficiency given by
    \begin{equation} \label{eq:time integral of torque constraint 2}
        \int_0^{T_2} \lVert \mathcal{T}_{\text{\scriptsize s}}^N(\tau)\rVert_{L^{\infty}} d\tau \le C\int_0^{T_2} \lVert \mathcal{T}_{\text{\scriptsize s}}^N(\tau)\rVert_{G^{\frac{p-1}{2}}_{\sigma}} d\tau \le \frac{\mi-\mf}{2} .
    \end{equation}
    We now examine each term of $\lVert \mathcal{T}_{\text{\scriptsize s}}^N(\tau)\rVert_{G^{\frac{p-1}{2}}_{\sigma}}$. We begin with
    \begin{equation} \label{eq:T_s^N first term}
        \lVert P^N(\us^N\cdot\nabla \ws^N)\rVert_{G^{\frac{p-1}{2}}_{\sigma}} \lesssim \lVert \us^N \rVert_{G^{\frac{p-1}{2}}_{\sigma}} \lVert \nabla \ws^N\rVert_{G^{\frac{p-1}{2}}_{\sigma}} \lesssim \lVert \us^N\rVert_{G^{\frac{p+1}{2}}_{\sigma}}^2 ,
    \end{equation}
    where we have used the algebra property of the Gevrey norm. In a similar manner, we also obtain
    \begin{equation} \label{eq:T_s^N third term}
        \lVert P^N(\ws^N\cdot\nabla \us^N)\rVert_{G^{\frac{p-1}{2}}_{\sigma}} \lesssim \lVert \us^N\rVert_{G^{\frac{p+1}{2}}_{\sigma}}^2 .
    \end{equation}
    Finally, we deal with the last term of $\mathcal{T}_{\text{\scriptsize s}}^N$ as 
    \begin{equation} \label{eq:T_s^N fourth term}
    \begin{aligned}
        &\left\Vert \nabla\times P^N \left(\frac{\ws^N}{\abs{\ws^N}}\times (\ws^N\times v^N)\right)\right\Vert_{G^{\frac{p-1}{2}}_{\sigma}} \\
        &\lesssim 
        \left\lVert \frac{\ws^N}{\abs{\ws^N}}\times (\ws^N\times v^N) \right\rVert_{G^{\frac{p}{2}}_{\sigma}} \\
        &\lesssim 
        \lVert \ws^N\rVert_{G^{\frac{p}{2}}_{\sigma}}^2 \lVert \frac{1}{\abs{\ws^N}} \rVert_{G^{\frac{p}{2}}_{\sigma}} \lVert v^N \rVert_{G^{\frac{p}{2}}_{\sigma}} \\
        &\lesssim 
        \lVert \us^N\rVert_{G^{\frac{p+1}{2}}_{\sigma}}^2 \left( e^{\sigma_0} + p!\left(\frac{\mf}{2}-C_0\sigma_0\right)^{-p-1} \right) \left(\lVert \us^N \rVert_{G^{\frac{p+1}{2}}_{\sigma}} + \lVert \un^N \rVert_{G^{\frac{p+1}{2}}_{\sigma}} \right) ,
    \end{aligned}
    \end{equation}
    where we made use of~\eqref{eq:1/omega analytic norm 3}. Combining~\eqref{eq:T_s^N first term}--\eqref{eq:T_s^N fourth term} with the uniform-in-$N$ bounds~\eqref{eq:Gevrey energy estimate result}, we see that it is indeed possible to find a $T_2$ (independent of $N$, and depending on the data) that is small enough so that~\eqref{eq:time integral of torque constraint 2} is met. This implies that $\lvert \ws^N \rvert \ge \mf$ for all $x\in\T^3$ and all $0\le t\le T^*$, where $T^* = \min\{T_1,T_2\}$, closing the necessary a priori estimates.

    \section{Existence of a unique solution} \label{sec:existence of a unique, time-analytic solution}

    \subsection{Existence of a solution}
    We fix the value of $N$ (larger than $N^*$, defined following~\eqref{eq:triangle inequality for vorticity}), and take the inner product of~\eqref{eq:HVBK approximation} with $e^{ij\cdot x}$. This yields a system of ODEs for the functions $\wh{\us}^N(t,k)$ and $\wh{\un}^N(t,k)$, which have a unique solution according to standard ODE theory. Once this is established, it is only a matter of using the a priori estimates in Section~\ref{sec:a priori estimates} to obtain a solution of~\eqref{eq:HVBK Leray projected} as the limit of a subsequence of approximating solutions. 

    \subsection{Uniqueness of the solution}
    Starting from the same initial data, we assume that there are two solutions $(\us^1,\un^1)$ and $(\us^2,\un^2)$, and denote $\Phi_s := \us^1-\us^2$, $\Phi_n := \un^1-\un^2$, $\Ws := \ws^1 - \ws^2$, and $\Wn := \wn^1 - \wn^2$. Then, subtracting~\eqref{eq:HVBK Leray projected} for the two pairs of solutions yields
    \begin{equation} \label{eq:HVBK difference of solutions}
        \begin{aligned}
    		\partial_t \Phi_s 
            &= 
            -\Leray(\Phi_s\cdot\nabla \us^1) -\Leray(\us^2\cdot\nabla \Phi_s) -\rho_{\text{\scriptsize n}} \Leray \mathcal{F} \\
            \partial_t \Phi_n
            &= 
            -\Leray(\Phi_n\cdot\nabla \un^1) -\Leray(\un^2\cdot\nabla \Phi_n) + \rho_{\text{\scriptsize s}} \Leray \mathcal{F} , 
    	\end{aligned}
    \end{equation}
    where
    \begin{equation} \label{eq:delta F expression}
    \begin{aligned}    
        \mathcal{F} &= \frac{\ws^1}{\abs{\ws^1}}\times (\ws^1\times v^1) - \frac{\ws^2}{\abs{\ws^2}}\times (\ws^2\times v^2) \\
        &= \frac{\ws^1}{\abs{\ws^1}}\times(\ws^1\times(\Phi_n - \Phi_s)) + \frac{\ws^1}{\abs{\ws^1}}\times(\Ws\times v^2) + \frac{\Ws}{\abs{\ws^1}}\times(\ws^2\times v^2) \\ 
        &\quad + \left(\frac{1}{\abs{\ws^1}} - \frac{1}{\abs{\ws^2}}\right)\ws^2\times(\ws^2\times v^2) .
    \end{aligned}
    \end{equation}
    We remark that a simple energy-level estimate, i.e., in $L^2$, is not going to be sufficient here, since terms involving the gradient of $\Phi_s$ appear from the mutual friction part of~\eqref{eq:HVBK difference of solutions}, and there is no viscosity. Hence, we have to induce an artificial dissipation through the analytic norm, as was done in Section~\ref{sec:Gevrey norm energy estimate}. Defining $q:= p-\frac{1}{2}>2$ (so that Lemma~\ref{lem:nonlinear estimate} still applies), we take the inner product of~\eqref{eq:HVBK difference of solutions}$_1$ with $\Phi_s$ in $G^{\frac{q}{2}}_{\sigma(t)}$ and obtain
    \begin{equation} \label{eq:uniqueness superfluid estimate 1}
    \begin{aligned}
        \frac{1}{2}\frac{d}{dt}\lVert \Phi_s \rVert_{G^{\frac{q}{2}}_{\sigma}}^2 + \delta \lVert A^{\frac{1}{4}}\Phi_s \rVert_{G^{\frac{q}{2}}_{\sigma}}^2 
        &= -\langle \Phi_s, \Phi_s\cdot\nabla \us^1 \rangle_{G^{\frac{q}{2}}_{\sigma}} - \langle \Phi_s, \us^2\cdot\nabla \Phi_s \rangle_{G^{\frac{q}{2}}_{\sigma}} + \rho_{\text{\scriptsize n}}\langle \Phi_s, \mathcal{F} \rangle_{G^{\frac{q}{2}}_{\sigma}} \\
        &= \sum_{i=1}^6 I_i ,
    \end{aligned}
    \end{equation}
    where the last four terms of the sum correspond to the terms on the RHS of~\eqref{eq:delta F expression}. In what follows, we repeatedly make use of the fact that $G^{\frac{q}{2}}_{\sigma}$ is an algebra. For the first term, the Cauchy-Schwarz inequality yields
    \begin{equation} \label{eq:I_1 term}
    \begin{aligned}
        I_1 
        \lesssim \lVert \nabla \us^2 \rVert_{G^{\frac{q}{2}}_{\sigma}} \lVert \Phi_s \rVert_{G^{\frac{q}{2}}_{\sigma}}^2
        \lesssim \lVert \us^2 \rVert_{G^{\frac{p+1}{2}}_{\sigma}} \lVert \Phi_s \rVert_{G^{\frac{q}{2}}_{\sigma}}^2. 
    \end{aligned}
    \end{equation}
    The term~$I_2$ is accounted for in a similar manner, while also keeping in mind that $\nabla \Phi_s$ is equivalent to $\Ws$ insofar as the analytic (or any Sobolev/Lebesgue) norm is concerned. Hence,
    \begin{equation} \label{eq:I_2 term}
    \begin{aligned}
        I_2 
        &\lesssim \lVert \us^2 \rVert_{G^{\frac{q}{2}}_{\sigma}} \lVert \Phi_s \rVert_{G^{\frac{q}{2}}_{\sigma}} \lVert \nabla\Phi_s \rVert_{G^{\frac{q}{2}}_{\sigma}} \lesssim \lVert \us^2 \rVert_{G^{\frac{q}{2}}_{\sigma}} \left(\lVert \Phi_s \rVert_{G^{\frac{q}{2}}_{\sigma}}^2 + \lVert \Ws \rVert_{G^{\frac{q}{2}}_{\sigma}}^2 \right) . 
    \end{aligned}
    \end{equation}
    We proceed on to the terms arising out of $\mathcal{F}$. Making use of~\eqref{eq:1/omega analytic norm 3}, we have
    \begin{equation} \label{eq:I_3 term}
    \begin{aligned}
        I_3 \lesssim \lVert \Phi_s \rVert_{G^{\frac{q}{2}}_{\sigma}} \lVert \ws^1 \rVert_{G^{\frac{q}{2}}_{\sigma}}^2 \lVert \Phi_n - \Phi_s \rVert_{G^{\frac{q}{2}}_{\sigma}} \lesssim \lVert \ws^1 \rVert_{G^{\frac{q}{2}}_{\sigma}}^2 \left( \lVert \Phi_s \rVert_{G^{\frac{q}{2}}_{\sigma}}^2 + \lVert \Phi_n \rVert_{G^{\frac{q}{2}}_{\sigma}}^2 \right) , 
    \end{aligned}
    \end{equation}
    and
    \begin{equation} \label{eq:I_4 term}
    \begin{aligned}
        I_4 \lesssim \lVert \Phi_s \rVert_{G^{\frac{q}{2}}_{\sigma}} \lVert \ws^1 \rVert_{G^{\frac{q}{2}}_{\sigma}} \lVert \Ws \rVert_{G^{\frac{q}{2}}_{\sigma}} \lVert v^2 \rVert_{G^{\frac{q}{2}}_{\sigma}} \lesssim \lVert v^2 \rVert_{G^{\frac{q}{2}}_{\sigma}} \lVert \ws^1 \rVert_{G^{\frac{q}{2}}_{\sigma}} \left( \lVert \Phi_s \rVert_{G^{\frac{q}{2}}_{\sigma}}^2 + \lVert \Ws \rVert_{G^{\frac{q}{2}}_{\sigma}}^2 \right) .
    \end{aligned}
    \end{equation}
    The term $I_5$ leads to an almost identical estimate:
    \begin{equation} \label{eq:I_5 term}
    \begin{aligned}
        I_5 
        &\lesssim \lVert v^2 \rVert_{G^{\frac{q}{2}}_{\sigma}} \lVert \ws^2 \rVert_{G^{\frac{q}{2}}_{\sigma}} \left( \lVert \Phi_s \rVert_{G^{\frac{q}{2}}_{\sigma}}^2 + \lVert \Ws \rVert_{G^{\frac{q}{2}}_{\sigma}}^2 \right) .
    \end{aligned}
    \end{equation}
    To handle the last term in~\eqref{eq:uniqueness superfluid estimate 1}, we begin by writing 
    \begin{equation*}
        \frac{1}{\abs{\ws^1}} - \frac{1}{\abs{\ws^2}} = -\frac{\abs{\ws^1} - \abs{\ws^2}}{\abs{\ws^1}\abs{\ws^2}} = -\frac{\ws^1 + \ws^2}{\abs{\ws^1}\abs{\ws^2} \left(\abs{\ws^2} + \abs{\ws^1}\right)} \cdot \Ws .
    \end{equation*}
    Upon application of~\eqref{eq:1/omega analytic norm 3}, this time with $\abs{\omega}$ replaced with $\abs{\omega^1}+\abs{\omega^2}$, we obtain
    \begin{equation} \label{eq:I_6 term}
    \begin{aligned}
        I_6 &\lesssim \lVert v^2 \rVert_{G^{\frac{q}{2}}_{\sigma}} \lVert \ws^2 \rVert_{G^{\frac{q}{2}}_{\sigma}}^2 \left( \lVert \ws^1 \rVert_{G^{\frac{q}{2}}_{\sigma}} + \lVert \ws^2 \rVert_{G^{\frac{q}{2}}_{\sigma}} \right) \left( \lVert \Phi_s \rVert_{G^{\frac{q}{2}}_{\sigma}}^2 + \lVert \Ws \rVert_{G^{\frac{q}{2}}_{\sigma}}^2 \right) .
    \end{aligned}
    \end{equation}
    Combining~\eqref{eq:uniqueness superfluid estimate 1}--\eqref{eq:I_6 term}, and utilize the analytic norm control obtained in~\eqref{eq:Gevrey energy estimate result}, to arrive at
    \begin{equation} \label{eq:uniqueness superfluid estimate 2}
        \frac{1}{2}\frac{d}{dt}\lVert \Phi_s \rVert_{G^{\frac{q}{2}}_{\sigma}}^2 + \delta \lVert A^{\frac{1}{4}}\Phi_s \rVert_{G^{\frac{q}{2}}_{\sigma}}^2 \le C \left( \lVert \Phi_s \rVert_{G^{\frac{q}{2}}_{\sigma}}^2 + \lVert \Phi_n \rVert_{G^{\frac{q}{2}}_{\sigma}}^2 + \lVert \Ws \rVert_{G^{\frac{q}{2}}_{\sigma}}^2\right) ,
    \end{equation}
    where the constant $C$ depends on $q$ as well as the size of the initial data according to the estimates in Section~\ref{sec:Gevrey norm energy estimate}. Repeating the process with~\eqref{eq:HVBK difference of solutions}$_2$, we obtain
    \begin{equation} \label{eq:uniqueness normal estimate 2}
        \frac{1}{2}\frac{d}{dt}\lVert \Phi_n \rVert_{G^{\frac{q}{2}}_{\sigma}}^2 + \delta \lVert A^{\frac{1}{4}}\Phi_n \rVert_{G^{\frac{q}{2}}_{\sigma}}^2 \le C \left( \lVert \Phi_s \rVert_{G^{\frac{q}{2}}_{\sigma}}^2 + \lVert \Phi_n \rVert_{G^{\frac{q}{2}}_{\sigma}}^2 + \lVert \Ws \rVert_{G^{\frac{q}{2}}_{\sigma}}^2 + \lVert \Wn \rVert_{G^{\frac{q}{2}}_{\sigma}}^2 \right) .
    \end{equation}
    
    Next, we subtract the vorticity equations~\eqref{eq:HVBK reformulated}$_3$ for the two different solutions. We drop the Galerkin truncation index $N$ as well as the finite-dimensional projector $N$ in this process, since we have already constructed a solution. The resulting equation is
    \begin{equation} \label{eq:HVBK vorticity difference of solutions}
        \partial_t \Ws 
        = 
        -\Phi_s\cdot\nabla \ws^1 -\us^2\cdot\nabla \Ws + \Ws\cdot\nabla \us^1 + \ws^2\cdot\nabla \Phi_s -\rho_{\text{\scriptsize n}} \nabla\times \mathcal{F} .
    \end{equation}
    Following the analysis from above, we take the $G^{\frac{q}{2}}_{\sigma(t)}$ inner product of~\eqref{eq:HVBK vorticity difference of solutions} with $\Ws$, to obtain
    \begin{equation} \label{eq:uniqueness superfluid vorticity estimate 1}
    \begin{aligned}
        \frac{1}{2}\frac{d}{dt}\lVert \Ws \rVert_{G^{\frac{q}{2}}_{\sigma}}^2 + \delta \lVert A^{\frac{1}{4}}\Ws \rVert_{G^{\frac{q}{2}}_{\sigma}}^2 
        &= \langle \nabla\Ws, \Phi_s\otimes \ws^1 \rangle_{G^{\frac{q}{2}}_{\sigma}} + \langle \nabla\Ws, \us^2\otimes \Ws \rangle_{G^{\frac{q}{2}}_{\sigma}} - \langle \nabla\Ws, \Ws\otimes \us^1 \rangle_{G^{\frac{q}{2}}_{\sigma}} \\
        &\quad + \langle \Ws, \ws^2\cdot\nabla \Phi_s \rangle_{G^{\frac{q}{2}}_{\sigma}} + \rho_{\text{\scriptsize n}}\langle \nabla\Ws, \mathcal{F} \rangle_{G^{\frac{q}{2}}_{\sigma}} \\
        &= \sum_{i=1}^8 J_i ,
    \end{aligned}
    \end{equation}
    where, once again, the last four terms in the sum correspond to the RHS of~\eqref{eq:delta F expression}. The estimate of $J_1$ is a consequence of Lemma~\ref{lem:nonlinear estimate}, since
    \begin{equation} \label{eq:J_1 term}
    \begin{aligned}
        J_1 &\lesssim \lVert A^{\frac{1}{4}}\Ws \rVert_{G^{\frac{q}{2}}_{\sigma}} \left( \lVert A^{\frac{1}{4}}\Phi_s \rVert_{G^{\frac{q}{2}}_{\sigma}} \lVert \ws^1 \rVert_{G^{\frac{q}{2}}_{\sigma}} + \lVert \Phi_s \rVert_{G^{\frac{q}{2}}_{\sigma}} \lVert A^{\frac{1}{4}}\ws^1 \rVert_{G^{\frac{q}{2}}_{\sigma}} \right) \\
        &\lesssim \lVert \ws^1 \rVert_{G^{\frac{p}{2}}_{\sigma}} \left( \lVert A^{\frac{1}{4}}\Ws \rVert_{G^{\frac{q}{2}}_{\sigma}}^2 + \lVert A^{\frac{1}{4}}\Phi_s \rVert_{G^{\frac{q}{2}}_{\sigma}}^2 \right) .
    \end{aligned}
    \end{equation}
    The terms $J_2$ and $J_3$ result in identical estimates, except for the index of the superfluid velocity. Thus, we have
    \begin{equation} \label{eq:J_2, J_3 terms}
    \begin{aligned}
        J_2, J_3 &\lesssim \lVert A^{\frac{1}{4}}\Ws \rVert_{G^{\frac{q}{2}}_{\sigma}} \left( \lVert A^{\frac{1}{4}}\Ws \rVert_{G^{\frac{q}{2}}_{\sigma}} \lVert \us \rVert_{G^{\frac{q}{2}}_{\sigma}} + \lVert \Ws \rVert_{G^{\frac{q}{2}}_{\sigma}} \lVert A^{\frac{1}{4}}\us \rVert_{G^{\frac{q}{2}}_{\sigma}} \right) \\
        &\lesssim \lVert \us \rVert_{G^{\frac{p}{2}}_{\sigma}} \lVert A^{\frac{1}{4}}\Ws \rVert_{G^{\frac{q}{2}}_{\sigma}}^2 ,
    \end{aligned}
    \end{equation}
    where the generic velocity $\us$ is a stand-in for both $\us^1$ and $\us^2$, whichever is appropriate. For the term $J_4$, we use the algebra property of the analytic norm, which yields
    \begin{equation} \label{eq:J_4 term}
    \begin{aligned}
        J_4 &\lesssim \lVert \Ws \rVert_{G^{\frac{q}{2}}_{\sigma}} \lVert \ws^2 \rVert_{G^{\frac{q}{2}}_{\sigma}} \lVert \nabla\Phi_s \rVert_{G^{\frac{q}{2}}_{\sigma}}
        \lesssim \lVert \ws^2 \rVert_{G^{\frac{q}{2}}_{\sigma}} \lVert \Ws \rVert_{G^{\frac{q}{2}}_{\sigma}}^2 .
    \end{aligned}
    \end{equation}
    Now, we move on to the terms corresponding to the mutual friction. Using Lemma~\ref{lem:nonlinear estimate} and~\eqref{eq:1/omega analytic norm 3}, the first of these becomes
    \begin{equation} \label{eq:J_5 term}
    \begin{aligned}
        J_5 
        &\lesssim \lVert A^{\frac{1}{4}}\Ws \rVert_{G^{\frac{q}{2}}_{\sigma}} \left( \lVert A^{\frac{1}{4}}\ws^1 \rVert_{G^{\frac{q}{2}}_{\sigma}} \lVert \ws^1 \rVert_{G^{\frac{q}{2}}_{\sigma}} \lVert \Phi_n-\Phi_s \rVert_{G^{\frac{q}{2}}_{\sigma}} \right) \\
        &\quad + \lVert A^{\frac{1}{4}}\Ws \rVert_{G^{\frac{q}{2}}_{\sigma}} \left( \lVert \ws^1 \rVert_{G^{\frac{q}{2}}_{\sigma}}^2 \lVert \Phi_n-\Phi_s \rVert_{G^{\frac{q}{2}}_{\sigma}} \right) \\
        &\quad + \lVert A^{\frac{1}{4}}\Ws \rVert_{G^{\frac{q}{2}}_{\sigma}} \left( \lVert \ws^1 \rVert_{G^{\frac{q}{2}}_{\sigma}}^2 \lVert A^{\frac{1}{4}}(\Phi_n-\Phi_s) \rVert_{G^{\frac{q}{2}}_{\sigma}} \right) \\
        &\lesssim \lVert \ws^1 \rVert_{G^{\frac{p}{2}}_{\sigma}}^2 \left( \lVert A^{\frac{1}{4}}\Ws \rVert_{G^{\frac{q}{2}}_{\sigma}}^2 + \lVert A^{\frac{1}{4}}\Phi_s \rVert_{G^{\frac{q}{2}}_{\sigma}}^2 + \lVert A^{\frac{1}{4}}\Phi_n \rVert_{G^{\frac{q}{2}}_{\sigma}}^2 \right) . 
    \end{aligned}
    \end{equation}
    Just like in~\eqref{eq:J_2, J_3 terms}, the terms $J_6$ and $J_7$ yield almost identical estimates, except for the index of one of the vorticities (which we will omit). Following the same steps as for $J_5$, we get
    \begin{equation} \label{eq:J_6, J_7 term}
    \begin{aligned}
        J_6, J_7 
        &\lesssim \lVert \ws \rVert_{G^{\frac{p}{2}}_{\sigma}}^2 \lVert v^2 \rVert_{G^{\frac{p}{2}}_{\sigma}} \lVert A^{\frac{1}{4}}\Ws \rVert_{G^{\frac{q}{2}}_{\sigma}}^2  .
    \end{aligned}
    \end{equation}
    Finally, the treatment of the term $J_8$ proceeds similar to the term $I_6$. Indeed, we obtain
    \begin{equation} \label{eq:J_8 term}
    \begin{aligned}
        J_8 &\lesssim \lVert v^2 \rVert_{G^{\frac{p}{2}}_{\sigma}} \lVert \ws^2 \rVert_{G^{\frac{p}{2}}_{\sigma}}^2 \left( \lVert \ws^1 \rVert_{G^{\frac{p}{2}}_{\sigma}} + \lVert \ws^2 \rVert_{G^{\frac{p}{2}}_{\sigma}} \right) \lVert A^{\frac{1}{4}}\Ws \rVert_{G^{\frac{q}{2}}_{\sigma}}^2  .
    \end{aligned}
    \end{equation}
    From~\eqref{eq:uniqueness superfluid vorticity estimate 1}--\eqref{eq:J_8 term}, the resulting equation is
    \begin{equation} \label{eq:uniqueness superfluid vorticity estimate 2}
        \frac{1}{2}\frac{d}{dt}\lVert \Ws \rVert_{G^{\frac{q}{2}}_{\sigma}}^2 + \delta \lVert A^{\frac{1}{4}}\Ws \rVert_{G^{\frac{q}{2}}_{\sigma}}^2 \le C\left( \lVert A^{\frac{1}{4}}\Ws \rVert_{G^{\frac{q}{2}}_{\sigma}}^2 + \lVert A^{\frac{1}{4}}\Phi_s \rVert_{G^{\frac{q}{2}}_{\sigma}}^2 + \lVert A^{\frac{1}{4}}\Phi_n \rVert_{G^{\frac{q}{2}}_{\sigma}}^2 \right) .
    \end{equation}
    Repeating the above calculations with~\eqref{eq:HVBK reformulated}$_4$ (after dropping the finite-dimensional approximations) leads to a similar estimate for the normal fluid, namely
    \begin{equation} \label{eq:uniqueness normal vorticity estimate 2}
        \frac{1}{2}\frac{d}{dt}\lVert \Wn \rVert_{G^{\frac{q}{2}}_{\sigma}}^2 + \delta \lVert A^{\frac{1}{4}}\Wn \rVert_{G^{\frac{q}{2}}_{\sigma}}^2 \le C\left( \lVert A^{\frac{1}{4}}\Ws \rVert_{G^{\frac{q}{2}}_{\sigma}}^2 + \lVert A^{\frac{1}{4}}\Wn \rVert_{G^{\frac{q}{2}}_{\sigma}}^2 + \lVert A^{\frac{1}{4}}\Phi_s \rVert_{G^{\frac{q}{2}}_{\sigma}}^2 + \lVert A^{\frac{1}{4}}\Phi_n \rVert_{G^{\frac{q}{2}}_{\sigma}}^2 \right) .
    \end{equation}

    We now add~\eqref{eq:uniqueness superfluid estimate 2},~\eqref{eq:uniqueness normal estimate 2},~\eqref{eq:uniqueness superfluid vorticity estimate 2}, and~\eqref{eq:uniqueness normal vorticity estimate 2} to arrive at
    \begin{equation} \label{eq:uniqueness final estimate}
    \begin{aligned}
        &\frac{1}{2}\frac{d}{dt}\left( \lVert \Phi_s \rVert_{G^{\frac{q}{2}}_{\sigma}}^2 + \lVert \Phi_n \rVert_{G^{\frac{q}{2}}_{\sigma}}^2 + \lVert \Ws \rVert_{G^{\frac{q}{2}}_{\sigma}}^2 + \lVert \Wn \rVert_{G^{\frac{q}{2}}_{\sigma}}^2 \right) \\
        &\quad + \delta \left( \lVert A^{\frac{1}{4}}\Phi_s \rVert_{G^{\frac{q}{2}}_{\sigma}}^2 + \lVert A^{\frac{1}{4}}\Phi_n \rVert_{G^{\frac{q}{2}}_{\sigma}}^2 + \lVert A^{\frac{1}{4}}\Ws \rVert_{G^{\frac{q}{2}}_{\sigma}}^2 + \lVert A^{\frac{1}{4}}\Wn \rVert_{G^{\frac{q}{2}}_{\sigma}}^2 \right) \\
        &\le C\left( \lVert A^{\frac{1}{4}}\Phi_s \rVert_{G^{\frac{q}{2}}_{\sigma}}^2 + \lVert A^{\frac{1}{4}}\Phi_n \rVert_{G^{\frac{q}{2}}_{\sigma}}^2 + \lVert A^{\frac{1}{4}}\Ws \rVert_{G^{\frac{q}{2}}_{\sigma}}^2 + \lVert A^{\frac{1}{4}}\Wn \rVert_{G^{\frac{q}{2}}_{\sigma}}^2 \right) \\
        &\quad + C \left( \lVert \Phi_s \rVert_{G^{\frac{q}{2}}_{\sigma}}^2 + \lVert \Phi_n \rVert_{G^{\frac{q}{2}}_{\sigma}}^2 + \lVert \Ws \rVert_{G^{\frac{q}{2}}_{\sigma}}^2 + \lVert \Wn \rVert_{G^{\frac{q}{2}}_{\sigma}}^2 \right) .
    \end{aligned}
    \end{equation}
    We remark that $C$ is a constant that depends on $q$ and on the size of the initial data. Thus, by selecting $\delta = 2C$, we may absorb the first term on the RHS into the LHS. Subsequently, uniqueness is achieved using the Gr\"onwall inequality, since we have identical initial data for both solutions, i.e., 
    $\Phi_s(0) = \Phi_n(0) = \Ws(0) = \Wn(0) = 0$.
    
    \qed

    \section*{Acknowledgments}
    The author is grateful to Juhi Jang and Igor Kukavica for useful discussions.

    \appendix

    \section{Analytic norm of $A^{\frac{r}{2}} \frac{1}{\abs{\omega}}$}

    For any $r\in\R^+$, denoting by $\ceil{r}$ the least integer greater than or equal to $r$, we observe that
    \begin{equation} \label{eq:1/omega analytic norm 1}
    \begin{aligned}
        \lVert A^{\frac{r}{2}} \frac{1}{\abs{\omega}} \rVert_{G^{\frac{p}{2}}_{\sigma}} 
        &= \lVert A^{\frac{p+r}{2}} e^{\sigma A^{\frac{1}{2}}} \frac{1}{\abs{\omega}} \rVert_{L^2} = \lVert \sum_{n=0}^{\infty} \frac{\sigma^n}{n!}A^{\frac{n+p+r}{2}} \frac{1}{\abs{\omega}} \rVert_{L^2} \lesssim \sum_{n=0}^{\infty} \frac{\sigma_0^n}{n!} \lVert \frac{1}{\abs{\omega}} \rVert_{H^{n+\ceil{p}+\ceil{r}}} \\
        &\lesssim \sum_{n=0}^{\infty} \frac{\sigma_0^n}{n!} \left( \lVert \frac{1}{\abs{\omega}} \rVert_{L^2} + C_0^{n+\ceil{p}+\ceil{r}}\lVert D^{n+\ceil{p}+\ceil{r}} \frac{1}{\abs{\omega}} \rVert_{L^2} \right) \\
        &\lesssim e^{\sigma_0} + \sum_{n=0}^{\infty} \frac{(C_0\sigma_0)^n}{n!} \lVert D^{n+\ceil{p}+\ceil{r}} \frac{1}{\abs{\omega}} \rVert_{L^2} .
    \end{aligned}
    \end{equation}
    In the above estimates, $C_0$ is the Poincar\'e constant. To arrive at the last step, we made use of the fact that $\abs{\omega} \ge \mf$. The norm within the infinite sum may be estimated using the analyticity of $x^{-1}$ for $x$ bounded away from $0$. Indeed, we may expand $x^{-1}$ in a convergent Taylor series about some sufficiently large $M>0$. (For instance, $M > X_0 = 1 + \lVert\ws^0 \rVert_{G^{\frac{p}{2}}_{\sigma_0}} + \lVert\wn^0 \rVert_{G^{\frac{p}{2}}_{\sigma_0}}$ so that $\abs{\omega} \le \lVert \omega \rVert_{G^{\frac{p}{2}}_{\sigma_0}} < M$.) Thus,
    \begin{equation} \label{eq:Taylor series for x^-1/2}
        \frac{1}{\abs{\omega}} = \sum_{n=0}^\infty (-1)^n \frac{1}{M^{n+1}} \left( \abs{\omega} - M \right)^n . 
    \end{equation}
    This series converges (absolutely) for $0 < \abs{\omega} < 2M$, showing that $\frac{1}{\abs{\omega}}$ is analytic in the open disc centered at $M$ and of radius $M$. We may now complexify the function and apply Cauchy's differentiation formula to evaluate the higher order derivatives as
    \begin{equation} \label{eq:Cauchy diff formula}
        D^n\left( \frac{1}{z} \right) = \frac{n!}{2\pi i} \oint_C \frac{\frac{1}{z_1}}{(z_1 - z)^{n+1}} dz_1 .
    \end{equation}
    We set $z = \abs{\omega}$ and choose the integration contour to be a circle centered at $\abs{\omega}$ and passing through, say $\frac{\mf}{2}$ on the left end. Since $\abs{\omega}>\mf$, we obtain
    \begin{equation} \label{eq:nth derivative bound Cauchy}
        \abs{D^n\left( \frac{1}{\abs{\omega}} \right)} \le n! \frac{\frac{1}{\frac{\mf}{2}}}{\left(\abs{\omega} - \frac{\mf}{2}\right)^n} \lesssim \frac{2^n n!}{(\mf)^n} .
    \end{equation}

    Returning to~\eqref{eq:1/omega analytic norm 1}, we now utilize the estimates in~\eqref{eq:nth derivative bound Cauchy} which leads to
    \begin{equation} \label{eq:1/omega analytic norm 2}
    \begin{aligned}
        &\lVert A^{\frac{r}{2}} \frac{1}{\abs{\omega}} \rVert_{G^{\frac{p}{2}}_{\sigma}} \\ 
        &\lesssim e^{\sigma_0} + \sum_{n=0}^{\infty} \frac{(C_0\sigma_0)^n}{n!} \frac{2^{n+\ceil{p}+\ceil{r}}(n+\ceil{p}+\ceil{r})!}{(\mf)^{n+\ceil{p}+\ceil{r}}} \\
        &= e^{\sigma_0} + (C_0\sigma_0)^{-\ceil{p}-\ceil{r}} \sum_{n=0}^{\infty} \left(\frac{2C_0\sigma_0}{\mf}\right)^{n+\ceil{p}+\ceil{r}} (n+\ceil{p}+\ceil{r})(n+\ceil{p}+\ceil{r}-1)\dots (n+1) \\
        &= e^{\sigma_0} + (C_0\sigma_0)^{-\ceil{p}-\ceil{r}} \left. \sum_{n=0}^{\infty} \beta^{n} n(n-1)\dots (n-\ceil{p}-\ceil{r}+1) \right\vert_{\beta = \frac{2C_0\sigma_0}{\mf}} .
    \end{aligned}
    \end{equation}
    For $\beta<1$, or equivalently $\sigma_0 < \frac{\mf}{2C_0}$, the series converges uniformly (and absolutely). This allows us to evaluate the series explicitly,
    \begin{equation} \label{eq:1/omega analytic norm 3}
    \begin{aligned}
        \lVert A^{\frac{r}{2}} \frac{1}{\abs{\omega}} \rVert_{G^{\frac{p}{2}}_{\sigma}} 
        &\lesssim e^{\sigma_0} + (C_0\sigma_0)^{-\ceil{p}-\ceil{r}} \beta^{\ceil{p}+\ceil{r}} \left. \sum_{n=0}^{\infty} \left(\frac{d}{d\beta}\right)^{\ceil{p}+\ceil{r}} \beta^{n} \right\vert_{\beta = \frac{2C_0\sigma_0}{\mf}} \\
        &= e^{\sigma_0} + (C_0\sigma_0)^{-\ceil{p}-\ceil{r}} \beta^{\ceil{p}+\ceil{r}} \left. \left(\frac{d}{d\beta}\right)^{\ceil{p}+\ceil{r}} \frac{1}{1-\beta} \right\vert_{\beta = \frac{2C_0\sigma_0}{\mf}} \\
        &= e^{\sigma_0} + \left( \frac{2}{\mf} \right)^{\ceil{p}+\ceil{r}} \frac{(\ceil{p}+\ceil{r})!}{\left(1-\frac{2C_0\sigma_0}{\mf^2} \right)^{\ceil{p}+\ceil{r}+1}} \\
        &\lesssim e^{\sigma_0} + \frac{(\ceil{p}+\ceil{r})!}{\left(\frac{\mf}{2}-C_0\sigma_0 \right)^{\ceil{p}+\ceil{r}+1}} .
    \end{aligned}
    \end{equation}

	\addtocontents{toc}{\protect\setcounter{tocdepth}{2}}
	
	\bibliographystyle{alpha}
	\bibliography{references}
	
\end{document}